\documentclass[12pt]{amsart}

\usepackage{setspace}
\usepackage{amsmath, amssymb,amsthm, amsfonts}
\newtheorem{thm}{Theorem}
\newtheorem{lma}{Lemma}
\newtheorem{prop}{Proposition}
\newtheorem{cor}{Corollary}

\theoremstyle{definition}

\theoremstyle{remark}
\newtheorem{remark}{Remark}

\newcommand{\tr}{\mbox{tr}}

\renewcommand{\div}{\mbox{div}}

\newcommand{\Ric}{\mbox{Ric}}
\newcommand{\R}{\mathbb R}

\newcommand{\be}{\begin{equation}}
\newcommand{\ee}{\end{equation}}
\newcommand{\bee}{\begin{equation*}}
\newcommand{\eee}{\end{equation*}}

\def\b{\beta\mathbb{}}
\def\na{\nabla}

\def\la{\langle}
\def\ra{\rangle}

\def\Pi{\displaystyle{\mathbb{II}}}

\def\tf{\tilde{f}}

\def\a{\alpha}

\def\tr{\tilde{r}}

\def\C{\mathbb{C}}

\begin{document}

\title[]{gradient steady K\"{a}hler Ricci solitons with non-negative Ricci curvature and integrable scalar curvature}

\author{Pak-Yeung Chan}
\address{School of Mathematics,
University of Minnesota, Minneapolis, MN 55455, USA} \email{chanx305@umn.edu}

\maketitle
\markboth{Pak-Yeung Chan} {Gradient steady K\"{a}hler Ricci solitons with $\Ric\geq 0$ and $S$ $\in$ $L^1$}

\begin{abstract} We study the non Ricci flat gradient steady K\"{a}hler Ricci soliton with non-negative Ricci curvature and weak integrability condition of the scalar curvature $S$, namely $\underline{\lim}_{r\to \infty} r^{-1}\int_{B_r} S=0$, and show that it is a quotient of $\Sigma\times \C^{n-1-k}\times N^k$, where $\Sigma$ and $N$ denote the Hamilton's cigar soliton and some compact K\"{a}hler Ricci flat manifold respectively. As an application, we prove that any non Ricci flat gradient steady K\"{a}hler Ricci soliton with $\Ric\geq 0$, together with subquadratic volume growth or $\limsup_{r\to \infty} rS<1$ must have universal covering space isometric to $\Sigma\times \C^{n-1-k}\times N^k$.

\end{abstract}

\section{Introduction}
Let $(M^m,g)$ be a real $m$ dimensional Riemannian manifold and $X$ be a smooth vector field on $M$, the triple $(M,g,X)$ is said to be a Ricci soliton if there is a constant $\lambda$ such that the following equation is satisfied
\be\label{eq-RS-1}
\Ric+\dfrac{1}{2}L_{X}g=\lambda g,
\ee
where $Ric$ and $L_{X}$ denote the Ricci curvature and Lie derivative with respect to $X$ respectively. A Ricci soliton is called shrinking (steady, expanding) if $\lambda>0$ $(=0, <0)$. It is said to be gradient if $X$ can be chosen such that $X=\nabla f$ for some smooth function $f$ on $M$.  The soliton is called complete if $(M,g)$ is complete as a Riemannian manifold.

 Ricci soliton is a self similar solution of the Ricci flow and often arises as a blow up limit of the Ricci flow near its singularities. It is closely related to the singularities models of the Ricci flow introduced by Hamilton \cite{Hamilton-1995}. The classification of Ricci soliton would give us a better understanding on the singularities formation of the Ricci flow.

 Ricci solitons can also be viewed as an extension of the Einstein metric $Ric=\lambda g$. Bakry and Emery first introduced the Bakry Emery Ricci curvature $Ric_f:=Ric+\nabla^2 f$ in \cite{BakryEmery-1985}. The Bakry Emery Ricci curvature is one of the most important geometric quantities in the theory of smooth metric measure spaces and appears in other branches of mathematics like probability (see \cite{Lott-2003}). Together with the fact that $L_{\nabla f}g=2\nabla^2 f$, the gradient Ricci solitons equation \eqref{eq-RS-1} can be rewritten as
\be\label{eq-RS-2}
\Ric_{f}=Ric+\nabla^2 f=\lambda g,
\ee
which is a natural generalization of the Einstein metric.

In \cite{Deruelle-2012}, Deruelle proved that any complete non-flat gradient steady Ricci soliton with non-negative sectional curvature and scalar curvature $S$ $\in L^1(M,g)$ is isometric to a quotient of $\Sigma\times\R^{m-2}$, where $\Sigma$ denotes the Hamilton's cigar soliton. Later Catino-Mastrolia-Monticelli \cite{CatinoMastroliaMonticelli-2016} weakened the integrability condition of $S$ to
\be\label{weakly integrable S}
\liminf_{r\to \infty}\frac{1}{r}\int_{B_r(p_0)}S=0,
\ee
for some $p_0$ $\in M$ (see also \cite{MunteanuSungWang-2017} by Munteanu-Sung-Wang for a different proof).  Since $S\geq 0$ (see Section 2), It is clear that the above condition is true independent on the base point $p_0$, i.e. (\ref{weakly integrable S}) holds for some $p_0$ in $M$ if and only if it holds for all $p_0$ in $M$.

\begin{thm}\cite{Deruelle-2012}, \cite{CatinoMastroliaMonticelli-2016}, \cite{MunteanuSungWang-2017}\label{nonnegative Sec and integrable scalar curvature} Let $(M,g,f)$ be a real $m$ dimensional non-flat complete gradient steady Ricci soliton with non-negative sectional curvature. Suppose in addition that the scalar curvature $S$ satisfies (\ref{weakly integrable S}), then the universal cover of $(M,g)$ is isometric to $\Sigma\times\R^{m-2}$, where $\Sigma$ denotes the Hamilton's cigar soliton.
\end{thm}
\begin{remark}It is not difficult to see from the proof of Theorem \ref{nonnegative Sec and integrable scalar curvature} that Any non-flat gradient steady K\"{a}hler Ricci soliton with non-negative bisectional curvature and $S$ satisfying (\ref{weakly integrable S}) is a quotient of $\Sigma\times\C^{n-1}$.
\end{remark}
It was shown by Hamilton \cite{Hamilton-1995}, Ivey \cite{Ivey-1993} and Chen \cite{Chen-2009} that any real 3 dimensional complete gradient shrinking or steady Ricci soliton must have non-negative sectional curvature. However, this significant feature doesn't hold true for higher dimensions, Feldman, Ilmanen and Knopf \cite{FeldmanIlmanenKnopf} constructed some shrinkers with Ricci curvature being negative in some directions (see also \cite{Cao-1996} by Cao who constructed a steadier on anticanonical line bundle on $\mathbb{CP}^n$ which doesn't have non-negative bisectional curvature). It is natural to ask whether we can classify steady Ricci soliton under weaker curvature condition.
 Deng and Zhu \cite{DengZhu-2015} showed that any complete Ricci non-negative gradient steady K\"{a}hler Ricci soliton with average of scalar curvature over large ball decaying faster than linear rate must be Ricci flat. It would be interesting to know more about the K\"{a}hler steadier with non-negative Ricci curvature. In \cite{Deruelle-2012}, Deruelle proved the following local splitting theorem:
\begin{thm}\label{local splitting}\cite{Deruelle-2012} Let $(M,g,f)$ be a real $m$ dimensional complete gradient steady Ricci soliton with $\Ric\geq 0$ and $S>0$. Suppose the following conditions are satisfied:
\begin{enumerate}
\item $S$ is integrable, i.e. $S$ $\in L^1(M,g)$;
\item $|Rm|\to 0$ as $r\to \infty$;
\item $|\na f|^2S\geq 2\Ric(\na f, \na f)$,
\end{enumerate}

then $M\setminus A$ is locally isometric to $\Sigma\times\R^{m-2}$, where $A:=\{\na f=0\}$ and $\Sigma$ is the Hamilton's cigar soliton.
\end{thm}
\begin{remark}Condition 3 in the above theorem is automatic if $(M,g,f)$ is a gradient K\"{a}hler Ricci soliton with $\Ric\geq 0$.

\end{remark}

We shall generalize Theorem \ref{nonnegative Sec and integrable scalar curvature} and Theorem \ref{local splitting} under the K\"{a}hler condition. Here is the main result of this paper:

\begin{thm}\label{nonnegative Ric and integrable scalar curvature} Let $(M,g,f)$ be a complex n dimensional complete non Ricci flat gradient steady K\"{a}hler Ricci soliton with $\Ric\geq 0$ and $n\geq 2$. Suppose the scalar curvature $S$ satisfies (\ref{weakly integrable S}), i.e.
$$\liminf_{r\to \infty} \frac{1}{r}\int_{B_r}S=0 ,$$
then it is isometric to a quotient of  $\Sigma\times \C^{n-1-k}\times N^k$, where $\Sigma$ and $N$ denote the Hamilton's cigar soliton and some simply connected compact K\"{a}hler Ricci flat manifold  of complex dimension $k$ respectively.
\end{thm}
The result is no longer true if one allows $\liminf_{r\to \infty} \frac{1}{r}\int_{B_r}S>0$. Indeed, let $\Sigma_2$ be the positively curved $U(2)$ invariant soliton on $\C^2$ constructed by Cao \cite{Cao-1996} and $\mathbb{T}^{n-2}$ be any flat Tori of complex dimension $n-2$. $\liminf_{r\to \infty} \frac{1}{r}\int_{B_r}S>0$ for $\Sigma_2\times \mathbb{T}^{n-2}$ but its universal cover is not isometric to $\Sigma\times \C^{n-1-k}\times N^k$.

One difficulty we encounter is that in real dimension $m\geq 4$, the strong maximum principle for the Ricci tensor of Hamilton \cite{Hamilton-1986}, Cao \cite{Cao-2004} and the splitting theorem of soliton by Guan, Lu and Xu \cite{GuanLuXu-2015} are not available in the absence of non-negative sectional or bisectional curvature condition. Moreover, the classical Cheeger Gromoll splitting theorem (\cite{CheegerGromoll-1971} and \cite{Li-2012}) cannot be applied directly as the soliton under consideration has no line. Thanks to the observation by Deruelle in \cite{Deruelle-2012}, in order to split the manifold, one suffices to show that $\na f$ is an eigenvector of the Ricci tensor. Motivated by the arguments in \cite{CatinoMastroliaMonticelli-2016} and \cite{MunteanuSungWang-2017}, we will prove this by an integration by part argument.

    In view of Theorem \ref{nonnegative Sec and integrable scalar curvature}, one may ask when $\C^{n-1-k}\times N^k$ is flat, i.e. $k=0$. Under the assumptions of the previous theorem, we give a necessary and sufficient condition for the flatness of $\C^{n-1-k}\times N^k$.
\begin{cor}
Let $(M^n, g, f)$ be a complex n dimensional complete non Ricci flat gradient steady K\"{a}hler Ricci soliton with $\Ric\geq 0$ and $n\geq 2$. Suppose that
$$\liminf_{r\to \infty} \frac{1}{r}\int_{B_r}S=0.$$
For n=2, it is isometric to a quotient of $\Sigma\times\C$. For $n\geq 3$, $M$ is isometric to a quotient of $\Sigma\times \C^{n-1}$ if and only if $|Rm|\to 0$ as $r\to \infty$.
\end{cor}

The integrability condition (\ref{weakly integrable S}) is closely related to the volume growth of the manifold. Indeed, it was shown in \cite{Deruelle-2012} (see also \cite{CatinoMastroliaMonticelli-2016}) that for a complete gradient steady Ricci soliton with $\Ric\geq 0$ and scaling convention (\ref{scaling is 1}), the scalar curvature $S$ must satisfy
\be \label{volume and integrability of S}
\frac{1}{V(B_r(p))}\int_{B_r(p)}S\leq \frac{m}{r},
\ee
for all $r>0$ and $p$ $\in M^m$. With the above inequality, Catino, Mastrolia and Monticelli \cite{CatinoMastroliaMonticelli-2016} showed that any non-flat complete gradient steady Ricci soliton with non-negative sectional curvature and subquadratic volume growth is a quotient of $\Sigma\times \R^{m-2}$. Motivated by their result, we prove an analog in the K\"{a}hler case with $\Ric\geq 0$ using Theorem \ref{nonnegative Ric and integrable scalar curvature}.

\begin{cor}\label{subquadratic vol growth}
Let $(M^n, g, f)$ be a complex n dimensional complete non Ricci flat gradient steady K\"{a}hler Ricci soliton with $\Ric\geq 0$. Suppose the volume of geodesic ball is of subquadratic growth, i.e. $V(B_r)=o(r^2)$,
then the universal covering space of $M$ is isometric to $\Sigma\times \C^{n-1-k}\times N^k$, where $N$ is a simply connected compact K\"{a}hler Ricci flat manifold of complex dimension $k$.
\end{cor}
Recently, there have been lots of researches about the classification of Ricci solitons according to the decay rate of the scalar curvature. For example, Brendle \cite{Brendle-2013} showed that any real 3 dimensional complete non-flat and non-collapsed gradient steady Ricci soliton is the Bryant soliton (see also \cite{Brendle-2014}). Deng and Zhu \cite{DengZhu-2016},\cite{DengZhu-2018} later generalized Brendle's result and classified real 3 dimensional complete gradient Ricci steadier under $S\leq Cr^{-1}$. Munteanu, Sung and Wang \cite{MunteanuSungWang-2017} proved that any real $m$ dimensional non-flat gradient steadier with non-negative sectional curvature and decay rate of the scalar curvature faster than linear rate is isometric to a quotient of $\Sigma\times\R^{m-2}$. Lately, Deng and Zhu \cite{DengZhu-2018} generalized the result in \cite{MunteanuSungWang-2017}:
\begin{thm}\cite{DengZhu-2018} \label{DZ fast scalar decay} Let $(M,g,f)$ be a real $m$ dimensional complete non-flat gradient steady Ricci soliton with non-negative sectional curvature and the scaling convention (\ref{scaling is 1}). There exists a constant $\varepsilon=\varepsilon (m)>0$ depending only on $m$ such that if $S$ satisfies
$$rS\leq \varepsilon $$
near infinity, then the universal covering space of $M$ is isometric to $\Sigma\times\R^{m-2}$.
\end{thm}

Using a result by Catino, Mastrolia and Monticelli \cite{CatinoMastroliaMonticelli-2016} and Corollary \ref{subquadratic vol growth}, we can have a sharp dimension free bound for the $\varepsilon$ in Theorem \ref{DZ fast scalar decay}.

\begin{thm}\label{sec >0 and rS small}Let $(M,g,f)$ be a real $m$ dimensional complete non-flat gradient steady Ricci soliton with non-negative sectional curvature and the scaling convention (\ref{scaling is 1}). In addition, we assume that
$$\displaystyle\limsup_{r \to \infty}    rS < 1.$$
Then $M$ is isometric to a quotient of $\Sigma\times\R^{m-2}$ and $\displaystyle\limsup_{r \to \infty}    rS=0.$
\end{thm}
\begin{thm}\label{Ric >0, Kahler and rS small}Let $(M,g,f)$ be a complex $n$ dimensional complete non-Ricci flat gradient steady K\"{a}hler Ricci soliton with non-negative Ricci curvature and the scaling convention (\ref{scaling is 1}). In addition, we assume that
$$\displaystyle\limsup_{r \to \infty}    rS < 1.$$
Then $M$ is isometric to a quotient of $\Sigma\times \C^{n-1-k}\times N^k$ and $\displaystyle\limsup_{r \to \infty}    rS=0$, where $N$ is a simply connected compact K\"{a}hler Ricci flat manifold.
\end{thm}
If one allows $\limsup_{r \to \infty}   rS\leq 1$, then both Theorems \ref{sec >0 and rS small} and \ref{Ric >0, Kahler and rS small} will not be true. The counter example for the real case is the 3 dimensional Bryant soliton and for the K\"{a}hler case is the positively curved $U(2)$ invariant example constructed by Cao on $\C^2$ \cite{Cao-1996}, both satisfy $\lim_{r\to \infty} rS=1$ but they are not the quotient of $\Sigma\times\R$ or $\Sigma\times \C$. Higher dimensional counter examples can be obtained by taking product with flat torus of suitable dimensions.

The paper is organized as follows. In Section 2, we introduce the basic preliminaries needed in the subsequent sections. In Section 3, we prove Theorem \ref{nonnegative Ric and integrable scalar curvature} assuming a proposition in Section 4. In Section 4, we study the geometry of $\Sigma\times N$ ($N$ is complete Ricci flat) with quotient satisfying (\ref{weakly integrable S}) and prove a proposition needed in the previous section. Lastly, we show Theorems \ref{sec >0 and rS small} and \ref{Ric >0, Kahler and rS small} in Section 5.

{\sl Acknowledgement}: The author would like to express deep gratitude to his advisor Prof. Jiaping Wang for his constant support, guidance and encouragement. The author is also grateful to Prof. Huai-Dong Cao, Prof. Ovidiu Munteanu and Prof. Luen-Fai Tam for their helpful comments and interests in this work. The author is indebted to Fei He, Shaochuang Huang, Man-Chun Lee, Man-Shun Ma, Dekai Zhang and Bo Zhu for valuable discussions over the last several years. Part of this work was written while the author was visiting Yau Mathematical Sciences Center of Tsinghua University. He would like to thank her for the hospitality. The author was partially supported by NSF grant DMS-1606820.

\section{preliminaries and notations}

Let $(M,g)$ be a connected smooth Riemannian manifold and $f$ be a smooth function on $M$. $(M,g,f)$ is said to be a gradient steady Ricci soliton with potential function $f$ if
\be\label{eq-RS-22}
\Ric+\na^2 f=0.
\ee
A K\"{a}hler manifold $(M, g, J)$ is a gradient steady K\"{a}hler Ricci soliton if $M$ satisfies (\ref{eq-RS-22}) for some smooth function $f$ and complex structure $J$ on $M$ (see \cite{Chowetal-2007}). A steady soliton is complete if $(M,g)$ is a complete Riemannian manifold. We fix a point $p_0$ $\in M$ and denote the distance function w.r.t. $g$ from $p_0$ by $r=r(x)=d(x,p_0)$. A normalized geodesic $\gamma$ $: \R\to M$ is called a line if for all real numbers $a$ and $b$ with $a\leq b$, $\gamma\mid_{[a,b]}$ is distance minimizing.
Given any Riemannian manifold $(\widetilde{N},g_{\widetilde{N}})$, $S_{\widetilde{N}}$ refers to the scalar curvature of $\widetilde{N}$ w.r.t. $g_{\widetilde{N}}$. For simplicity, we omit the subscript $\widetilde{N}$ in $S_{\widetilde{N}}$ when $\widetilde{N}=M$ and $g_{\widetilde{N}}=g$. Let $\beta$ $\in\R$ and $h$ be any function on $M$, $h=o(r^{\beta})$ means that $\lim_{r\to \infty} r^{-\beta}h=0$. We also adopt the Einstein summation convention in this paper, i.e. any repeated index is interpreted as a sum over that index.

Ricci soliton is a self similar solution to the Ricci flow. Given a complete gradient steady Ricci soliton, let $g(t):=\varphi_t^*g$, where $t$ $\in \R$ and $\varphi_t$ is the flow of $\na f$ with  $\varphi_0= id$. Then $g(t)$ is a solution to the Ricci flow:

\be\label{eq-RF-1}
\begin{split}
\dfrac{\partial g(t)}{\partial t}&=-2\Ric(g(t))\\
g(0)&=g
\end{split}
\ee

 It was shown by Chen \cite{Chen-2009} that any complete ancient solution to the Ricci flow must have nonnegative scalar curvature. Using strong maximum principle, we see that any complete gradient steady Ricci soliton must have positive scalar curvature $S>0$ unless it is Ricci flat (see also \cite{Zhang-2009}). It is also known that any compact steady Ricci soliton is Ricci flat \cite{Chowetal-2007} and hence any non Ricci flat complete gradient steady Ricci soliton is non-compact.

Hamilton \cite{Hamilton-1995} showed that for a complete gradient steady Ricci soliton, there exists a constant $c$ such that $|\na f|^2+S=c$ on $M$ ($c\geq 0$ since $S \geq 0$). When $c>0$ (in particular if $g$ is not Ricci flat), upon scaling the metric by a constant, we have
\be\label{scaling is 1}
|\na f|^2+S=1.
\ee
We shall adopt the above scaling convention (\ref{scaling is 1}) throughout this paper. The following identities are well known for gradient Ricci steadier (see \cite{Hamilton-1995}, \cite{Chowetal-2007}, \cite{Cao-2010}):
\be \label{lap f and S}
\Delta f+S=0,
\ee
\be\label{eqn of S}
\Delta S-\langle\na f, \na S\rangle=-2|\Ric|^2
\ee
and
\be \label{grad S and Ric}
2\Ric(\na f)=\na S.
\ee

The earliest non-Einstein gradient Ricci soliton $\Sigma$ was constructed by Hamilton in \cite{Hamilton-1988}. It is called cigar soliton and is a real $2$ dimensional complete gradient steady soliton defined on $\R^2$. $\Sigma$ is rotationally symmetric with positive sectional curvature. In the standard coordinate of $\R^2$, its metric is given by (see \cite{Cao-2010})
$$g_{\Sigma}=\frac{4(dx^2+dy^2)}{1+x^2+y^2},$$
together with the function $f(x,y)=-\log(1+x^2+y^2)$ and the complex structure on $\C$, $(\Sigma, g_{\Sigma}, f)$ is a complete gradient steady K\"{a}hler Ricci soliton. It is also the unique (up to scaling) real 2 dimensional non-flat complete gradient steady Ricci soliton (see \cite{Chowetal-2007}, \cite{BernsteinMettler-2015} and ref. therein). See \cite{Cao-2010} and \cite{Chowetal-2007} for more properties of $\Sigma$ and examples of Ricci solitons.

It was shown in \cite{CaoChen-2012} and \cite{CarrilloNi-2009} (see also \cite{Deruelle-2012}) that for a complete gradient steady Ricci soliton with $\Ric>0$ and $S$ attaining maximum ( or $\Ric\geq 0$ with $\limsup_{r\to \infty}S<\max_M S$), then there exist $a\in (0,1)$ and $D>0$ such that
$$r+D\geq -f\geq a r-D \text{  on  } M.$$
We first prove a similar bound for $f$ under different conditions which suffice for the arguments in later sections. Similar estimate was also obtained independently by Deng and Zhu \cite{DengZhu-2018} without non Ricci flat condition, instead $\Ric\geq 0$ on $M$ and $-f$ being equivalent to $r$ are assumed.
\begin{prop}\label{properness of f}
Let $(M,g,f)$ be a real $m$ dimensional complete non-Ricci flat gradient steady Ricci soliton with $\Ric\geq 0$ outside some compact subset of $M$. Further suppose that $S\to 0$ as $r\to \infty$. Then for all $\a$ $\in(0,1)$, there exists $D>0$ such that
\be\label{-f equiv to dist}
r+D\geq -f\geq \a r-D \text{  on  } M,
\ee
where $r$ is the distance function from a fixed reference point $p_0$ $\in M$. In particular, $\displaystyle\lim_{r\to \infty}\frac{-f}{r}=1$.
\end{prop}
\begin{proof}The upper bound of $-f$ follows from (\ref{scaling is 1}) and $|\na f|\leq 1$. For the lower bound, let $\delta$ be a small positive constant to be chosen. Since $|\na f|^2+S\equiv 1$ and $S\to 0$ at infinity, there is a compact subset $K$ of $M$ such that $p_0$ $\in K$ and on $M\setminus K$, $\Ric\geq 0$ and
\be\label{lower bdd of na f2}
|\na f|\geq \frac{1}{1+\delta}.
\ee
Let $\psi_t$ be the flow of $\frac{\na f}{|\na f|^2}$ with $\psi_0$ be the identity map. Let $q$ $\in M\setminus K$ and for small $t\geq 0$, by (\ref{lower bdd of na f2}),
\be\label{distance v.s. t}
d(\psi_t(q),q)\leq \int_0^t\frac{1}{|\na f|(\psi_s(q))}ds\leq (1+\delta)t.
\ee
By short time existence of O.D.E., $\psi_t(q)$ exists as long as it is in $M\setminus K$. Therefore we can define $T$ as follows
\be
T:=\sup\{a: \psi_t(q)\in M\setminus K \text{  for all  } t \in [0,a]\}.
\ee
Obviously, $T=T(q)>0$ by compactness of $K$. For $0\leq t<T$
\be\label{value of f along integral curve}
f(\psi_t(q))-f(q)=\int_0^t\la\na f,\dot{\psi}_s(q)\ra ds=\int_0^t 1 ds=t.
\ee
We first show that $T<\infty$. Suppose not, then $T=\infty$ and by (\ref{value of f along integral curve}), there is a sequence of $t_k\to \infty$ such that $\psi_{t_k}(q)\to \infty$ as $k\to \infty$. But by (\ref{grad S and Ric})
\begin{eqnarray*}
S(\psi_{t_k}(q))-S(q)&=&\int_0^{t_k}\la \na S,\dot{\psi}_s(q)\ra ds\\
&=&\int_0^{t_k}\la \na S,\frac{\na f(\psi_s(q))}{|\na f|^2}\ra ds\\
&=&\int_0^{t_k} \frac{2\Ric(\na f,\na f)}{|\na f|^2}ds\\
&\geq&0.
\end{eqnarray*}
Hence $S(\psi_{t_k}(q))\geq S(q)>0$ and $\lim_{k\to \infty}S(\psi_{t_k}(q))\neq 0$, contradicting to our assumption that $S=o(1)$. We proved that $T<\infty$ and $\psi_{T}(q)\in K$. By (\ref{distance v.s. t}), $d(\psi_T(q),q)\leq (1+\delta)T$.
$$r(q)=d(p_0,q)\leq d(\psi_T(q),q)+d(\psi_T(q),p_0)\leq (1+\delta)T+ \text{diam}K,$$
where $\text{diam}K$ is the diameter of the subset $K$. We have
\begin{eqnarray*}
-f(q)&=&T-f(\psi_T(q))\\
&\geq& T-\sup_K |f|\\
&\geq&\frac{1}{1+\delta}r(q)-\frac{\text{diam}K}{1+\delta}-\sup_K |f|.
\end{eqnarray*}
(\ref{-f equiv to dist}) follows by choosing $\delta>0$ small enough such that $\frac{1}{1+\delta}\geq \a$. $-r^{-1}f\to 1$ as $r\to \infty$ is now a consequence of (\ref{-f equiv to dist}).
\end{proof}

\section{Proof of theorem \ref{nonnegative Ric and integrable scalar curvature}}
To start with, we recall a result on the kernel of the Ricci tensor of steady soliton satisfying (\ref{weakly integrable S}). It was proved in \cite{MunteanuSungWang-2017} in the real case with non-negative sectional curvature. However, the argument also works well in the K\"{a}hler case with non-negative Ricci curvature. For the sake of completeness, we include the proof of the result here.
\begin{prop}\label{ker of Ric}\cite{MunteanuSungWang-2017}
Let $(M,g,f)$ be a complex n dimensional complete non Ricci flat gradient steady K\"{a}hler Ricci soliton with $\Ric\geq 0$. Suppose that
$$\liminf_{r\to \infty} \frac{1}{r}\int_{B_r}S=0.$$  Then $S^2\equiv 2|\Ric|^2$ and the null space $E$ of the Ricci tensor is a subbundle of the tangent bundle $TM$ with real rank $2n-2$.
\end{prop}
\begin{proof}The argument is essentially due to \cite{MunteanuSungWang-2017}. Let $\lambda_i, i=1,2,\ldots, 2n$ be the eigenvalues of the Ricci tensor. By $J$ invariance of $\Ric$, we may assume $\lambda_i=\lambda_{n+i} , i=1,2,\ldots, n$ and $0\leq\lambda_1\leq \lambda_2\leq\ldots\leq \lambda_n$. Hence

\begin{eqnarray*}
S-2\lambda_i&=&S-2\lambda_{n+i}\\
&=&\big(\displaystyle\sum_{j\neq i}^{2n}\lambda_j\big)-\lambda_i\\
&=&\big(\displaystyle\sum_{j\neq i}^{2n}\lambda_j\big)-\lambda_{n+i}\\
&=&\displaystyle\sum_{j\neq i, n+i}^{2n}\lambda_j\geq 0.
\end{eqnarray*}

From this, we know that
$$2|\Ric|^2=\sum_{j=1}^{2n}2\lambda_j^2\leq \sum_{j=1}^{2n}\lambda_jS=S^2,$$
with equality holds at a point $p$ iff $\lambda_n=\lambda_{2n}=\frac{S}{2}$ at $p$ iff the dimension of the null space of $\Ric$ at $p$ is $2n-2$. We are going to show $2|\Ric|^2\equiv S^2$ on $M$.
Let $\varphi$ be a non-negative cut off function which $\equiv 1$ on $B_R(p_0)$, $\equiv 0$ outside $B_{2R}(p_0)$ and $|\nabla \varphi|\leq \frac{c}{R}$. We know that by the contracted second Bianchi identity $2\div(\Ric)=\na S$,
\begin{eqnarray*}
0 &\leq& \displaystyle\int_M \varphi^2(S^2-2|\Ric|^2)\\
&=& \displaystyle\int_M \varphi^2(-S\Delta f+2R_{ij}f_{ij})\\
&=& \displaystyle\int_M \varphi^2(\langle\nabla S,\nabla f\rangle-2R_{ij,j}f_i)\\
& & +\displaystyle\int_M 2\varphi S \langle\nabla \varphi,\nabla f\rangle-\displaystyle\int_M 4\varphi R_{ij}f_i\varphi_j\\
&\leq& \dfrac{c}{R}\displaystyle\int_{B_{2R}(p_0)} S.
\end{eqnarray*}
By condition (\ref{weakly integrable S}), one can pick a sequence of $R_k \to \infty$ such that R.H.S. goes to zero as $k \to \infty$, we show that $S^2=2|\Ric|^2$ everywhere. It is not difficult to see from the previous argument that $\Ric$ only has two distinct eigenvalues, one is $0$ with multiplicity $2n-2$, another one is $\frac{S}{2}$ with multiplicity $2$, result follows.\\
\end{proof}
Since we do not impose any condition on the sign of bisectional curvature, the non-triviality of the kernel of the Ricci tensor doesn't suffice for the splitting. Motivated by the local splitting result in \cite{Deruelle-2012} (see Theorem \ref{local splitting}), we show that $\na f$ is always an eigenvector of $\Ric$ which eventually leads to the splitting of $M$.
\begin{prop}\label{grad f is eigenvector}Let $(M,g,f)$ be a complex n dimensional complete gradient steady K\"{a}hler Ricci soliton with $\Ric\geq 0$. Suppose that
$$\liminf_{r\to \infty} \frac{1}{r}\int_{B_r}S=0 ,$$
then $|\na f|^2S=2\Ric(\na f, \na f)$ on $M$. In particular if $M$ is not Ricci flat, then it is isometric to a quotient of $\Sigma \times N$, where $\Sigma$ and $N$ denote the cigar soliton and some simply connected complete K\"{a}hler Ricci flat manifold respectively.
\end{prop}
\begin{proof} We are done if $g$ is Ricci flat, so we can assume $\Ric$ is not identically zero. Since $\Ric\geq 0$ and the curvature tensor is $J$ invariant, we have $|\na f|^2S\geq 2\Ric(\na f, \na f)$. Let $Q:=\sqrt{f^2+1}\geq 1$. Then $\na Q = Q^{-1} f \na f$. Let $\phi$ $\in$ $C^{\infty}_c(M)$ be any smooth compactly supported function on $M$.
\begin{eqnarray*}
0 &\leq& \int_{M}\phi^2Q^{-1}(|\na f|^2S- 2\Ric(\na f, \na f))\\
&=&\int_{M}\phi^2Q^{-1} f_i f_i S-\int_{M}2\phi^2Q^{-1}R_{ij}f_i f_j\\
&=:& (I)+(II)
\end{eqnarray*}
Using integration by part, we have
\begin{eqnarray*}
(I)&=&-\int_M 2\phi\phi_iQ^{-1}f f_iS + \int_M \phi^2Q^{-3} f f_i f f_i S\\
&&-\int_M \phi^2Q^{-1}f f_{ii}S-\int_M\phi^2Q^{-1}f f_iS_i\\
&=&-\int_M 2\phi\phi_iQ^{-1}f f_iS + \int_M \phi^2Q^{-3}f^2f_i f_i S\\
&&+\int_M \phi^2Q^{-1}f S^2-\int_M\phi^2Q^{-1}f f_iS_i,
\end{eqnarray*}
where we use the fact that (\ref{lap f and S}) $\Delta f+S=0$. Similarly, using (\ref{grad S and Ric}) $\na S=2 \Ric(\na f)$ and the contracted second Bianchi identity $2\div \Ric=\na S$, we see that
\begin{eqnarray*}
(II)&=& \int_M 4\phi\phi_i Q^{-1}f f_jR_{ij}-\int_M2\phi^2Q^{-3}f f_i f f_jR_{ij}\\
& &+\int_M2\phi^2Q^{-1}f R_{ij,i}f_j+\int_M2\phi^2Q^{-1}f R_{ij}f_{ji}\\
&=&\int_M 2\phi\phi_i Q^{-1}f S_i-\int_M 2\phi^2Q^{-3}f^2 f_if_jR_{ij}\\
& &+\int_M\phi^2Q^{-1}f f_jS_j-\int_M2\phi^2Q^{-1}f |\Ric|^2\\
&=&\int_M 2\phi\phi_i Q^{-1}f S_i-\int_M 2\phi^2Q^{-3}f^2 f_if_jR_{ij}\\
& &+\int_M\phi^2Q^{-1}f f_jS_j-\int_M\phi^2Q^{-1}f S^2,\\
\end{eqnarray*}
we also use the identity $S^2=2|\Ric|^2$  (see \cite{MunteanuSungWang-2017} and Proposition \ref{ker of Ric}). Hence, we have
\begin{eqnarray*}
&& \int_{M}\phi^2Q^{-1}(|\na f|^2S- 2\Ric(\na f, \na f))\\
&=&-\int_M 2\phi\phi_iQ^{-1}f f_iS+\int_M 2\phi\phi_i Q^{-1}f S_i\\
&&+ \int_M \phi^2Q^{-3}f^2(f_if_i S-2f_if_jR_{ij})\\
&=&-\int_M 2\phi\phi_iQ^{-1}f f_iS+\int_M 2\phi\phi_i Q^{-1}f S_i\\
&&+ \int_{M} \phi^2Q^{-3}f^2(|\na f|^2 S-2\Ric(\na f, \na f)).
\end{eqnarray*}
Since $Q^{-1}-Q^{-3}f^2=Q^{-3}$, we know that
\begin{eqnarray*}
&&\int_{M}\phi^2Q^{-3}(|\na f|^2S- 2\Ric(\na f, \na f))\\
&=&-\int_M 2\phi\phi_iQ^{-1}f f_iS+\int_M 2\phi\phi_i Q^{-1}f S_i.
\end{eqnarray*}
Now we take $0\leq \phi\leq 1$ be a cut off function $\equiv 1 $ on $B_R$, vanishes outside $B_{2R}$ and $|\na \phi|\leq \frac{c}{R}$.
\begin{eqnarray*}
|\int_M 2\phi\phi_iQ^{-1}f f_iS|&\leq& \int_{B_{2R}\setminus B_R} \frac{2c}{R}Q^{-1}|f|S\\
&\leq& \frac{2c}{R}\int_{B_{2R}\setminus B_R} S.
\end{eqnarray*}
Since $\Ric\geq 0$, $|\na S|\leq 2|\Ric|\leq cS$.

\begin{eqnarray*}
|\int_M 2\phi\phi_i Q^{-1}f S_i|&\leq&\int_{B_{2R}\setminus B_R}\frac{2c}{R}Q^{-1}|f| |\na S|\\
&\leq& \frac{c_1}{R}\int_{B_{2R}\setminus B_R} S.
\end{eqnarray*}
All in all, there is a positive constant $c_2$ independent of $R$ such that
\begin{eqnarray*}
0&\leq& \int_{M\cap B_R}Q^{-3}(|\na f|^2S- 2\Ric(\na f, \na f))\\
&\leq& \frac{c_2}{R}\int_{B_{2R}\setminus B_R} S.
\end{eqnarray*}
Using the condition $\liminf_{r\to \infty} \frac{1}{r}\int_{B_r}S=0$, we may pick a sequence of $R_k \to \infty$ such that R.H.S. goes to zero as $k\to \infty$, we conclude that $|\na f|^2S=2\Ric(\na f, \na f)$ on $M$.
We now proceed to prove the splitting of $M$. By $\Ric\geq 0$ and $J$ invariance of $\Ric$, we have for any tangent vector $v$ with $|v|_g=1$,
$$2\Ric(v,v)=\Ric(v,v)+\Ric(Jv,Jv)\leq S.$$
Hence whenever $\na f\neq 0$, $2\Ric(\frac{\na f}{|\na f|},\frac{\na f}{|\na f|})=S$, we deduce that $\na f$ is an eigenvector with eigenvalue equal to $\frac{S}{2}$ and thus $\na f$ is always perpendicular to the nullspace of $\Ric$. Let $E$ be the nullspace of $\Ric$, it is a smooth subbundle of the tangent bundle of real rank $2n-2$ (\cite{MunteanuSungWang-2017} and Proposition \ref{ker of Ric}). Suppose at $p$, $\na f\neq 0$, the tangent space at $p$ decomposes orthogonally as $T_pM$ $=$ $E_p$ $\oplus_{\perp}$ $span\{\na f, J\na f\}$. Let $X$ be a smooth section of $E$ defined locally near $p$ and $Y$ be any smooth vector field defined near $p$, then $JX$ is also a smooth section of $E$. At $p$
\begin{eqnarray*}
\la\na_Y X,\na f\ra &=& Y\la X, \na f\ra- \la X, \na_Y\na f\ra\\
&=& \Ric(X, Y)\\
&=& 0.
\end{eqnarray*}
Similarly, $\na_Y JX \perp \na f$, thus $\na_Y X(p)$ is in $E_p$. If $\na f=0$ at $p$, by real analyticity of $g$ (see \cite{Kotschwar-2013} and ref. therein), $\{\na f=0\}=\{S=1\}$ has no interior point in $M$ (indeed if $p$ is an interior point, then by (\ref{eqn of S}) $0=\Delta S(p)-\langle\na f, \na S\rangle(p)=-2|\Ric|^2(p)$,  which is absurd). We may find a sequence $p_k \to p$ with $\na f(p_k)\neq 0$,
$$\Ric(\na_Y X)(p)=\lim_{k\to \infty}\Ric(\na_Y X)(p_k)=0.$$
From this, we conclude that $E$ is invariant under parallel translation. By de Rham splitting theorem (see \cite{KobayashiNomizu-1969}) and the classification of real 2 dimensional complete gradient steady Ricci solitons (see \cite{Chowetal-2007}, \cite{BernsteinMettler-2015} and ref. therein), the universal cover of $M$ splits like $\Sigma\times N$ for some K\"{a}hler Ricci flat $N$.
\end{proof}
\begin{proof}[Proof of Theorem \ref{nonnegative Ric and integrable scalar curvature}]
By Proposition \ref{grad f is eigenvector}, the universal covering space of $M$ splits isometrically as $\Sigma\times N$ for some simply connected complete K\"{a}hler Ricci flat $N$. By de Rham decomposition theorem (see \cite{KobayashiNomizu-1969}), $N$ is isometric to $\C^{n-1-k}\times N_1$ with $N_1$ being a product of irreducible K\"{a}hler Ricci flat manifolds. It remains to show $N_1$ is compact. By Proposition \ref{compactness of N} (will be proved in the coming section), $\R^{2n-2-2k}\times N_1\cong_{\text{ isom }} \R^l\times Q$, for some simply connected compact Ricci flat manifold $Q$. Since both $N_1$ and $Q$ have no line, we conclude that $2n-2-2k=l$ and $N_1$ is diffeomorphic to the compact $Q$.
\end{proof}

\section{Geometry of $\Sigma\times N/\sim$}
In this section, we study the geometry of the quotient manifold $M$ $= \Sigma\times N/\sim$ with scalar curvature $S$ satisfying (\ref{weakly integrable S}), where $\Sigma$ and $N$ denote the cigar soliton and some real $m-2$ dimensional simply connected complete (not necessarily K\"{a}hler) Ricci flat manifold respectively. The main goal of this section is to prove the following:

\begin{prop}\label{compactness of N}
Let $M^m$ $= \Sigma\times N/\sim$, for some simply connected complete Ricci flat manifold $N$.
Suppose that on $M$
$$\liminf_{r\to \infty} \frac{1}{r}\int_{B_r}S=0 ,$$
then there exist positive constants $C$ and $\a$ $\in (0,1)$ such that
$$C^{-1}e^{-\frac{r}{\a}}\leq S\leq Ce^{-\a r} \text{  on  } M.$$
Moreover, $N$ is isometric to $\R^{m-2-k}\times Q^k$, where $Q$ is some simply connected compact Ricci flat manifold.
\end{prop}
\begin{remark} It can be seen from the above proposition that $M$ must have bounded curvature and $S$ is integrable. Using an estimate in \cite{Deruelle-2012}, the level sets of $f$ (function constructed in Lemma \ref{existence of potential f}) have uniformly bounded diameter, hence $\a$ in the above proposition can indeed be chosen to be $1$. Alternatively, $\alpha=1$ also follows from the curvature estimates in \cite{Chan-2019} and \cite{MunteanuSungWang-2017}.
\end{remark}
To prepare for the proof of Proposition \ref{compactness of N}, we recall some basic properties of $\Sigma$ (see \cite{Chowetal-2007}). Let $\tr$ and $\tf$ be the distance function of $\Sigma$ from its origin and the potential function respectively. In the geodesic polar coordinate, the metric is given by
$$g_{\Sigma}=d\tr^2+4\tanh^2(\frac{\tr}{2})d\theta^2.$$
We also have $\tf=\tf(\tr)=-2\log\cosh{\frac{\tr}{2}}$ and the scalar curvature
\be \label{scalar curvature of cigar}
S_{\Sigma}=\frac{1}{\cosh^2(\frac{\tr}{2})}=e^{\tf}>0.
\ee
Let $\rho: \Sigma\times N\to M$ and $\pi: \Sigma\times N\to \Sigma$ be the Riemannian covering and the projection into the first factor respectively. $\tr\circ \pi$ and $\tf\circ \pi$ are functions defined on $\Sigma\times N$. By abuse of notation, we shall not distinguish  $\tr\circ \pi$ from $\tr$, $\tf\circ \pi$ from $\tf$, namely for all $(a, b)\in$ $\Sigma\times N$, $\tr\circ \pi (a, b)$ and $\tf\circ \pi (a, b)$ will be written as $\tr (a, b)$ and $\tf (a, b)$ respectively.
\begin{lma}\label{existence of potential f} Let M be the manifold as in Proposition \ref{compactness of N}.
There is a smooth function $f$ on $M$ such that $f \circ \rho=\tf$. With this $f$, $(M,g,f)$ is a complete gradient steady Ricci soliton.
\end{lma}
\begin{proof}
Let $(a, b)$ and $(c, d)$ $\in$ $\Sigma\times N$ such that $\rho(a, b)=\rho(c, d)$. Since $N$ is Ricci flat,
$$S_{\Sigma}(a)=S_{\Sigma\times N}(a, b)=S_{\Sigma\times N}(c, d)=S_{\Sigma}(c).$$

By (\ref{scalar curvature of cigar}), we conclude that $\tr(a, b)=\tr(a)=\tr(c)=\tr(c, d)$ and $\tf(a, b)=\tf(c, d)$. $\tf$ respects the quotient map $\rho$ and thus induces a map $f: M\to \R$ such that $f \circ \rho=\tf$. $(M,g,f)$ is a gradient steady Ricci soliton then follows from the facts that $\rho$ is a local isometry and $\tf$ is a potential function for the steady soliton $\Sigma\times N$.
\end{proof}

\begin{lma} \label{compactness of level sets of f} Let f be the function as in Lemma \ref{existence of potential f}. The level sets $\Sigma_t:=\{f=t\}$ are connected compact embedded hypersurfaces in $M$ for all $t<0$.
\end{lma}
\begin{proof} Note that $0=\max_{M}f=\max_{\Sigma\times N}\tf$. $M$ has $\Ric\geq 0$ and thus both $f$ and $\tf$ are concave functions with
$$\{f=0\}=\{\na f=0\} \text{  and  } \{\tf=0\}=\{\na \tf=0\}.$$
For $t<0$, $\Sigma_t:=\{f=t\}$ are embedded hypersurfaces and complete w.r.t. the induced metric from $(M,g)$. Since $\rho^{-1}(\Sigma_t)=\{\tf=t\}$ is diffeomorphic to $\mathbb{S}^1\times N$, $\Sigma_t$ is connected for all $t<0$.
Let $\psi_t$ be the flow of $\frac{\na f}{|\na f|^2}$ with $\psi_0$ be the identity map.
Using the level set flow $\psi_t$, we see that $\Sigma_t =\{f=t\}$ are diffeomorphic to each other for all $t<0$. Moreover, $\psi_t(\Sigma_{-2})=\Sigma_{t-2}$, for $t$ $\in [0,1]$. Therefore, it suffices to show that $\Sigma_{-2}$ is compact. Assume by contradiction that $\Sigma_{-2}$ is not compact. On $\{\tf\leq -1\}\subseteq \Sigma\times N$, by (\ref{scaling is 1}) and (\ref{scalar curvature of cigar})
$$|\na\tf|^2=1-e^{\tf}\geq 1-e^{-1}.$$
Using $\rho_*\na \tf= \na f$, there exists $\delta>0$ such that on $\{f\leq -1\}$
\be\label{lower bdd of na f}
|\na f|\geq \delta.
\ee

Let $v\in T\Sigma_{-2}$, then
\begin{eqnarray*}
\frac{\partial}{\partial t}\psi_t^*g(v,v)&=&\psi_t^*(L_{\frac{\na f}{|\na f|^2}}g)(v,v)\\
&=&2\na^2 f(\psi_{t*}v,\psi_{t*}v)|\na f|^{-2}\\
&=&-2\Ric(\psi_{t*}v,\psi_{t*}v)|\na f|^{-2}\\
&\leq& 0.
\end{eqnarray*}
Since the Ricci curvature of $M$ is bounded, there is a constant $C>0$ such that for $t$ $\in [0,1]$
\be\label{level set metric comparison}
Cg_0\leq g_t\leq g_0
\ee
where $g_t:=\psi_t^*g$ on $\Sigma_{-2}$. Let $B^t_R(q)$ be the intrinsic ball of $(\Sigma_{t}, g)$ with radius $R$ centered at $q$. It is not difficult to see that for $t<0$
\be\label{intrinsic extrinsic ball}
B^t_R(q)\subseteq \Sigma_{t}\cap B_R(q), \text{  for  } q \in \Sigma_{t},
\ee
where $B_R(q)$ is the geodesic ball in the ambient manifold $(M,g)$. Fix any $q_0$ in $\Sigma_{-2}$. Let $r_0=r(q_0):=d(q_0,p_0)$, $p_0$ $\in$ $M$ is a fixed reference point.
Hence by (\ref{lower bdd of na f}) for $t$ $\in [0,1]$
\be
d(\psi_t(q_0), q_0)\leq \int_0^t\frac{1}{|\na f|(\psi_s(q_0))}ds\leq \frac{t}{\delta}.
\ee
Next we show that for large $R>0$,
\be\label{moving center ball}
B_{R-\frac{1}{\delta}}(\psi_t(q_0))\subseteq B_{R+r_0}(p_0).
\ee
For all $z$ $\in$ $L.H.S.$
\begin{eqnarray*}
d(z,p_0)&\leq& d(z,\psi_t(q_0))+d(\psi_t(q_0),q_0)+d(q_0,p_0)\\
        &<& R-\frac{1}{\delta}+\frac{t}{\delta}+r_0\\
        &\leq& R+r_0,
\end{eqnarray*}
we proved the inclusion (\ref{moving center ball}). To proceed, we also need the following inclusion: for $t$ $\in$ $[0,1]$
\be \label{evolution of Ball}
\psi_t(B^{-2}_{R-\frac{1}{\delta}}(q_0))\subseteq B^{t-2}_{R-\frac{1}{\delta}}(\psi_t(q_0)),
\ee
where $B^t_R(q)$ is defined before (\ref{intrinsic extrinsic ball}). Let $z$ $\in$ $B^{-2}_{R-\frac{1}{\delta}}(q_0)$ and $\a$ $\subseteq\Sigma_{-2}$ be an intrinsic minimizing geodesic w.r.t $(\Sigma_{-2},g)$ joining $z$ and $q_0$. The length of $\psi_t\circ \a$ is given by
\begin{eqnarray*}
l_g(\psi_t\circ \a)&=&\int|d\psi_t(\dot{\a})|_g\\
&=&\int|\dot{\a}|_{\psi_t^*g}\\
&\leq&\int|\dot{\a}|_{g},
\end{eqnarray*}
where we use (\ref{level set metric comparison}) in the last inequality and (\ref{evolution of Ball}) follows. From (\ref{scalar curvature of cigar}), we see that there is a positive constant $C_0$ such that on $\{-2\leq f\}$,
\be\label{lower bound of S on annulus of f}
S\geq C_0.
\ee

We are going to derive a contradiction using the weak integrability condition (\ref{weakly integrable S}) of $S$. By (\ref{lower bound of S on annulus of f}), coarea formula, (\ref{moving center ball}), (\ref{intrinsic extrinsic ball}), (\ref{evolution of Ball}),
\begin{eqnarray*}
\int_{B_{R+r_0}(p_0)}S&\geq& \int_{B_{R+r_0}(p_0)\cap\{ -2\leq f\leq -1\}}S\\
&\geq& C_0\int_{B_{R+r_0}(p_0)\cap\{ -2\leq f\leq -1\}}\\
&=&C_0\int_0^1\int_{B_{R+r_0}(p_0)\cap\Sigma_{t-2}}\frac{1}{|\na f|}d\sigma_t dt\\
&\geq&C_0\int_0^1\int_{B_{R-\frac{1}{\delta}}(\psi_t(q_0))\cap\Sigma_{t-2}}d\sigma_t dt\\
&\geq&C_0\int_0^1\int_{B^{t-2}_{R-\frac{1}{\delta}}(\psi_t(q_0))}d\sigma_t dt\\
&\geq&C_0\int_0^1\int_{\psi_t(B^{-2}_{R-\frac{1}{\delta}}(q_0))}d\sigma_t dt\\
&=&C_0\int_0^1\int_{B^{-2}_{R-\frac{1}{\delta}}(q_0)}\psi_t^*d\sigma_t dt\\
&\geq&C_0C_1\int_0^1\int_{B^{-2}_{R-\frac{1}{\delta}}(q_0)}d\sigma_0 dt\\
&=&C_0C_1 A(B^{-2}_{R-\frac{1}{\delta}}(q_0)),
\end{eqnarray*}
where we also use (\ref{level set metric comparison}) in the last inequality, $d\sigma_t$ and $A(B^{-2}_{R-\frac{1}{\delta}}(q_0))$ denote the volume element of $\Sigma_{t-2}$ and the volume of the intrinsic geodesic ball $B^{-2}_{R-\frac{1}{\delta}}(q_0)$ in $(\Sigma_{-2},g)$ respectively. One can check readily that the induced metric on $\{\tf=-2\}$ is given by
$$4(1-e^{-2})d\theta^2+g_N,$$
where $g_N$ is the metric on $N$. $\{\tf=-2\}$ is obviously Ricci flat. $\rho^{-1}(\Sigma_{-2})=\{\tf=-2\}$ and thus $\Sigma_{-2}$ is covered by $\{\tf=-2\}$. Hence $\Sigma_{-2}$ is also Ricci flat (in particular $\Ric\geq 0$). If $\Sigma_{-2}$ is noncompact, then by Yau's lower volume estimate on noncompact manifolds with $\Ric\geq 0$ (see \cite{Yau-1985} and \cite{Li-2012}), there exists positive constant $C_2$ independent on all large $R$ such that
$$A(B^{-2}_{R-\frac{1}{\delta}}(q_0))\geq C_2(R-\frac{1}{\delta}).$$
From this we see that for all large $R$
\begin{eqnarray*}
\frac{1}{R+r_0}\int_{B_{R+r_0}(p_0)}S &\geq& C_0C_1 \frac{A(B^{-2}_{R-\frac{1}{\delta}}(q_0))}{{R+r_0}}\\
&\geq&C_0C_1C_2\frac{R-\frac{1}{\delta}}{R+r_0},
\end{eqnarray*}
contradicting to the weak integrability condition (\ref{weakly integrable S}) of $S$. We proved that $\Sigma_{-2}$ and hence $\Sigma_{t}$ are compact as long as $t<0$.
\end{proof}

\begin{lma}\label{compactness of sublevel set of f}
Let f be the function as in Lemma \ref{existence of potential f}. $\{f\geq-A\}$ is compact subset of $M$ for all $A>0$
\end{lma}

\begin{proof}
Suppose it is not true for some $A$, then by the completeness of $M$, $\{f\geq-A\}$ is unbounded and there exists a sequence $x_k\to \infty$ with $f(x_k)\geq -A$. Pick a sequence $y_k$ with $f(y_k)\to -\infty$, let $\gamma_k$ be a normalized minimizing geodesic joining $x_k$ to $y_k$, then $\gamma_k\cap \Sigma_{-A-1}\neq \phi$. By Lemma \ref{compactness of level sets of f}, $\Sigma_{-A-1}$ is compact and it implies that after passing to a subsequence, $\gamma_k$ converges to a line $\gamma_{\infty}$. By Cheeger Gromoll splitting theorem (see \cite{CheegerGromoll-1971} and \cite{Li-2012}), $M$ splits isometrically as $M_1\times \R$ for some complete manifold $M_1$. Let $(\a,\b)\in M_1\times \R$, then $S_{M_1\times \R}(\a,\b)=S_{M_1}(\a)>0$. Moreover, one have for all $R>0$
$$B_R(\a,\b)\supseteq B^{M_1}_{\frac{R}{\sqrt{2}}}(\a)\times B^{\R}_{\frac{R}{\sqrt{2}}}(\b).$$
Then
\begin{eqnarray*}
\int_{B_R(\a,\b)}S &\geq& \int_{B^{M_1}_{\frac{R}{\sqrt{2}}}(\a)\times B^{\R}_{\frac{R}{\sqrt{2}}}(\b)}S_{M_1}\\
&\geq& \sqrt{2}R\int_{B^{M_1}_{\frac{R}{\sqrt{2}}}(\a)}S_{M_1}\\
&\geq& \sqrt{2}R\int_{B^{M_1}_{\frac{1}{\sqrt{2}}}(\a)}S_{M_1},
\end{eqnarray*}
again contradicting to the integrability condition (\ref{weakly integrable S}) of $S$.
\end{proof}
\begin{lma}S decays exponentially.
\end{lma}
\begin{proof}
We first show that $S\to 0$ at infinity. Let $x_k$ $\in M$ $\to \infty$ as $k\to \infty$, $(a_k, b_k)\in \Sigma\times N$ such that $\rho(a_k, b_k)=x_k$. Then $\tr_k:=\tr(a_k, b_k)\to \infty$, where $\tr(a_k, b_k)$ is understood as $\tr\circ\pi(a_k, b_k)=\tr(a_k)$ as in the discussion right before Lemma \ref{existence of potential f}. Otherwise, it is bounded for some subsequence $k_j$, then $f(x_{k_j})$=$\tf(a_{k_j}, b_{k_j})=-2\log\cosh(\frac{\tr_{k_j}}{2})$ is bounded, by Lemma \ref{compactness of sublevel set of f}, $x_{k_j}$ has convergent subsequence, which is impossible. Hence $\tr_k:=\tr(a_k, b_k)\to \infty$. Then by (\ref{scalar curvature of cigar}), $S(x_k)=1/\cosh^2(\frac{\tr_k}{2})\to 0$ as $k\to \infty$. We deduce that $\lim_{r\to \infty}S=0$. By Proposition \ref{properness of f} (see also \cite{CarrilloNi-2009}), there exist $\alpha$ $\in (0,1)$ and $D>0$ such that,
\be
r(x)+D\geq -f(x)\geq \alpha r(x)-D \text{  on  } M
\ee
and
\be
\tr(a)+D\geq -\tf(a)\geq \alpha \tr(a)-D \text{  on  } \Sigma.
\ee
By the above two inequalities, we have for $x\in M$ and $(a, b)\in \Sigma\times N$ with $\rho(a, b)=x$,
$$\a r(x)-2D\leq \tr(a)\leq \frac{r(x)+2D}{\a}$$
and hence for some positive constants $C_1$ and $C_2$
\begin{eqnarray*}
S(x)=\frac{1}{\cosh^2(\frac{\tr(a)}{2})}&\leq& 4e^{-\tr(a)}\\
&\leq&  C_1e^{-\a r(x)},
\end{eqnarray*}
similarly for the lower bound,
\begin{eqnarray*}
S(x)&\geq& e^{-\tr(a)}\\
&\geq& C_2e^{-\frac{r(x)}{\a}}.
\end{eqnarray*}
\end{proof}

To finish the proof of Proposition \ref{compactness of N}, it remains to show  $N$ is isometric to $\R^l\times Q$, for some compact simply connected Ricci flat manifold $Q$.

\begin{proof}[Proof of Proposition \ref{compactness of N}] Fix any $t<0$,
 the induced metric on $\{\tf=t\}\subseteq\Sigma\times N$ is equal to
 $$4(1-e^{t})d\theta^2+g_N.$$
 We see that $\{\tf=t\}$ is isometric to $\mathbb{S}^1\times N$ and universally covered by $\R\times N$. By de Rham decomposition theorem (see \cite{KobayashiNomizu-1963}), $N$ is isometric to $\R^q\times \widetilde{N}$, where $q\geq 0$ and $\widetilde{N}=N_1\times N_2\times\cdots\times N_l$ is a product of irreducible simply connected Ricci flat manifolds $N_i$ with  $\text{dim}_{\R}N_i \geq 2$ $\forall i$. $\widetilde{N}$ has no line otherwise $N_i$ splits for some $i$, contradicting to its irreducibility.

Since $\rho$ is a Riemannian covering map, $\Sigma_t$ is covered by $\{\tf=t\}$ and is compact Ricci flat. We have by Cheeger Gromoll splitting theorem (see \cite{CheegerGromoll-1971} and \cite{Li-2012}) and the uniqueness of universal Riemannian covering space that
\be
\R\times N \cong_{\text{  isom  }}\R^{q+1}\times \widetilde{N} \cong_{\text{  isom  }}\R^k\times Q,
\ee
for some simply connected compact Ricci flat $Q$. Both $\widetilde{N}$ and $Q$ do not have a line, we must have $q+1=k$ and $\widetilde{N}$ is diffeomorphic to the compact $Q$. We are done with the proof of the proposition.
\end{proof}

\section{Proof of Theorems \ref{sec >0 and rS small} and \ref{Ric >0, Kahler and rS small}}
In this section, we will show Theorems \ref{sec >0 and rS small} and \ref{Ric >0, Kahler and rS small}. They essentially follow from the volume estimate on large geodesic balls:
\begin{prop}\label{volume estimate when rS small}
 Let $(M,g,f)$ be a real m dimensional non Ricci flat complete gradient steady Ricci soliton with $\Ric\geq 0$. Suppose there is a finite positive constant $l$ such that
\be\label{rS small condition}
\limsup_{r\to \infty} rS\leq l.
\ee
Then for all $p_0$ $\in M$ and $\varepsilon>0$, there exists positive constant $C$ such that for all large $R$,
$$V(B_R(p_0))\leq CR^{l+1+\varepsilon}.$$
In particular if $l<1$, then $M$ has subquadratic volume growth.
\end{prop}

\begin{proof}By Proposition \ref{properness of f} (see also \cite{CarrilloNi-2009}), there are $\a$ $\in (0,1)$ and $D>0$ such that
\be\label{f in volume estimate}
r+D\geq -f\geq \a r-D.
\ee
Hence $f$ attains maximum, adding a constant if necessary, we may assume $\max_M f=0$. Since $f$ is concave, we have
$$\{\na f=0\}=\{f=0\}.$$
By (\ref{f in volume estimate}), $\{-f=t\}$ are compact embedded hypersurfaces and diffeomorphic to each other for all $t>0$.
Let $\delta$ be a small positive number to be chosen later. By $S\to 0$ and (\ref{scaling is 1}), for all large $r$,
\be\label{na f big in volume estimate}
|\na f|^2\geq (1+\delta)^{-1}.
\ee
Let $n:=-\frac{\na f}{|\na f|}$ be the normal of $\{-f=t\}$, the second fundamental form of $\{-f=t\}$ w.r.t. $n$ is given by $\frac{\Ric}{|\na f|}$. We consider the flow of $-\frac{\na f}{|\na f|^2}$ and denote it by $\phi_s$ with $\phi_0=id$, then
$$\phi_s(\{-f=t\})=\{-f=t+s\}.$$
 Let $A(t)$ be the area of the level set $\{-f=t\}$. By the first variation formula, (\ref{na f big in volume estimate}), (\ref{f in volume estimate}) and (\ref{rS small condition}) (see \cite{Deruelle-2012}), for all large $t$,
\begin{eqnarray*}
A'(t)&=&\int_{\{-f=t\}}\frac{S-\Ric(n,n)}{|\na f|^2}\\
&\leq&\int_{\{-f=t\}}\frac{S}{|\na f|^2}\\
&\leq& (1+\delta)\int_{\{-f=t\}} S\\
&\leq& \frac{(1+\delta)^2l}{t}A(t),
\end{eqnarray*}
we used $\Ric\geq 0$ in the first inequality. The above differential inequality implies that there is a $t_1$ such that for $t\geq t_1$
\be
A(t)\leq \frac{A(t_1)}{t_1^{(1+\delta)^2l}}t^{(1+\delta)^2l}.
\ee
Integrate the above inequality w.r.t $t$, together with (\ref{na f big in volume estimate}) and (\ref{f in volume estimate}), we see that for all large $R$,
$$V(B_R(p_0))\leq CR^{(1+\delta)^2l+1}.$$
Result then follows by choosing $\delta>0$ small enough such that $(1+\delta)^2l<l+\varepsilon$.
\end{proof}
\begin{proof}[Proof of Theorems \ref{sec >0 and rS small} and \ref{Ric >0, Kahler and rS small}]
By Proposition \ref{volume estimate when rS small}, $M$ has subquadratic volume growth. Theorem \ref{sec >0 and rS small} then follows from \cite{CatinoMastroliaMonticelli-2016}. Theorem \ref{Ric >0, Kahler and rS small} is now a consequence of Corollary \ref{subquadratic vol growth}. By (\ref{volume and integrability of S}), we know that $S$ satisfies (\ref{weakly integrable S}).
By Proposition \ref{compactness of N}, $S$ decays exponentially and thus $\displaystyle\limsup_{r \to \infty}   rS=0.$
\end{proof}

\end{document}